\documentclass[a4paper]{amsart}
\usepackage{amssymb,latexsym}
\usepackage[utf8x]{inputenc}
\usepackage{amsmath,amsthm}
\usepackage{amsfonts,mathrsfs}
\usepackage{cleveref}
\usepackage{enumerate,units}
\usepackage[all]{xy}
\usepackage{graphicx}

\theoremstyle{plain}
\newtheorem{theorem}{Theorem}[section]
\newtheorem{corollary}[theorem]{Corollary}
\newtheorem{lemma}[theorem]{Lemma}
\newtheorem{proposition}[theorem]{Proposition}

\newtheorem{fact}[theorem]{Fact}
\newtheorem{hypothesis}{Hypothesis}
\newtheorem*{claim}{Claim}
\newtheorem*{theorem*}{Theorem}

\theoremstyle{remark}
\newtheorem{remark}[theorem]{Remark}

\newtheorem{question}[theorem]{Question}

\theoremstyle{definition}
\newtheorem{definition}[theorem]{Definition}

\numberwithin{equation}{section}
\newcommand{\forkindep}[1][]{%
  \mathrel{
    \mathop{
      \vcenter{
        \hbox{\oalign{\noalign{\kern-.3ex}\hfil$\vert$\hfil\cr
              \noalign{\kern-.7ex}
              $\smile$\cr\noalign{\kern-.3ex}}}
      }
    }\displaylimits_{#1}
  }
}

\newenvironment{claimproof}[1][\proofname]
  {%
    \proof[#1]%
  }
  {%
    \endproof%
  }

\newcounter{step}                   
    {\hfill $\clubsuit$             
     \vspace{7pt}\par}
\newcommand{\primef}{k}  
\newcommand{\chr}{\mathrm{char}}
  
\begin{document}
\title{A Generalization of von Staudt's Theorem on Cross-Ratios}
\author{Yatir Halevi and Itay Kaplan}

\thanks{The first author was partially supported by the  European Research Council grant 338821. The second author would like to thank the Israel Science Foundation for their support of this research (Grant no. 1533/14).}
\address{Einstein Institute of Mathematics, The Hebrew University of Jerusalem, Givat Ram 9190401, Jerusalem, Israel}
\email{yatir.halevi@mail.huji.ac.il}
\email{kaplan@math.huji.ac.il}

\begin{abstract}
A generalization of von Staudt's theorem that every permutation of the projective line that preserves harmonic quadruples is a projective semilinear map is given. It is then concluded that any proper supergroup of permutations of the projective semilinear group over an algebraically closed field of transcendence degree at least $1$ is 4-transitive.
\end{abstract}

\maketitle
\section{Introduction}
In his book \emph{Geometrie der Lage} (see \cite{vs}), first appearing in 1847, Karl Georg Christian von Staudt wanting to establish (real) projective geometry on an axiomatic approach defined a projectivity to be a permutation of the projective line $\mathbb{P}(\mathbb{R})=\mathbb{R}\cup\{\infty\}$ preserving \emph{harmonic quadruples}, i.e. quadruples of distinct elements having cross-ratio $-1$ (which can be defined strictly geometrically), where the cross-ratio of quadruple of distinct elements is
\[[a,b;c,d]=\frac{c-a}{c-b}\cdot \frac{d-b}{d-a}=-1.\]

He proved that a projectivity is a composition of a finite number of perspectivities (which are basic geomtric maps). It was noticed later that there is a small gap in von Staudt's reasoning, see \cite{coolidge, voelke} for a detailed historical background. 

Given a field $F$, a \emph{projectivity} (also known as a homography or a fractional linear transformation) of the projective line $\mathbb{P}(F)=F\cup \{\infty\}$ is an element of the group
\[PGL_2(F)=\left\{\frac{ax+b}{cx+d}: a,b,c,d\in F,\, ad-bc\neq 0\right\},\]
where the usual conventions when dealing with $\infty,0$ and fractions apply here. It is easy to see that projectivities preserve cross-ratios.

It was Schreier and Sperner who first proved in \cite[page 191]{sperner} that every permutation of $F\cup\{\infty\}$, where $F$ is a field of characteristic $\neq 2$, preserving harmonic quadruples, i.e. \[[a,b;c,d]=-1 \Longrightarrow [f(a),f(b);f(c),f(d)]=-1,\] is an element of 
\[P\Gamma L_2(F)=\left\{\frac{ax^\sigma+b}{cx^\sigma+d}:a,b,c,d\in F,\, ad-bc\neq 0, \, \sigma\in Aut(F)\right\}.\] 
Any result in this spirit is now called a \emph{von Staudt theorem}. 

Over the years this theorem was generalized by relaxing the assumptions on $F$, for instance for $F$ a skew-field or a ring with some additional assumptions, see the introduction in \cite{havlicek} for a survey of results in that direction.

In this paper, we follow a generalization in a different direction. Hoffman eliminated the restriction on the characteristic of the field and replaced $-1$ with any field element which is fixed by $Aut(F)$, see \cite{hoffman}. Our main result is Corollary \ref{C:main} from the text:
\begin{theorem*}
Let $F$ be a field, $\primef$ its prime field and $\emptyset\neq O\subseteq F\setminus \{0,1\}$ which is $Aut(F)$-invariant. If
\begin{enumerate}
\item $\primef(O)\subsetneq F$ and
\item if $\chr(F)=2$ then $F$ is perfect and $|F|>4$,
\end{enumerate}
then the subgroup of permutations $f$ of $F\cup \{\infty\}$ satisfying \[[a,b;c,d]\in O\Longleftrightarrow [f(a),f(b);f(c),f(d)]\in O\] is exactly $P\Gamma L_2(F)$.

\end{theorem*}

The motivation for seeking such a generalization came from infinite symmetric groups, and some model theory. It is well known that for a cardinal $\kappa$, the closed subgroups (in the product topology) of the infinite symmetric group $S_\kappa$, for a cardinal $\kappa$, in the product topology correspond exactly to automorphisms groups of first-order structures. 
Thus finding closed supergroups of such groups, sheds light on the the first order theory of such structures.  

In \cite{itaypierre}, the second author and Pierre Simon proved that the affine groups $AGL_n(\mathbb{Q})$ (for $n\geq 2$) and the projective linear groups $PGL_n(\mathbb{Q})$ (for $n\geq 3$) are maximal closed in $S_\omega$. They ask whether it is true that $P\Gamma L_2(F)$ is maximal closed, for an algebraically closed field $F$ of transcendence degree greater that $1$. The aim of this paper is a step towards answering this question.

Bogomolov and Rovinsky proved that $P\Gamma L_n(F)$ is maximal closed for $n\geq 3$ and any field $F$, see \cite{bogrov}. The reason for the distinction between $n=2$ and $n\geq 3$ is that by the fundamental theorem of projective geometry, $P\Gamma L_n(F)$ (for $n\geq 3$) is exactly the collineation group of $\mathbb{P}^{n-1}(F)$. On the other hand, for $n=2$, since $\mathbb{P}^1(F)$ is the projective line, all the points are collinear. 

If $P\Gamma L_2(F)$ were not maximal closed, a proper supergroup of it must preserve one out of a known family of relations, two of them being quaternary relations (see \cite{itaypierre} for details). The aim is to show that it can not preserve any member of this family of relations. In this paper, using the above theorem, we conclude that if $F$ is an algebraically closed field of transcendence degree at least $1$, then any group of permutation of $F\cup\{\infty\}$ properly containing $P\Gamma L_2(F)$ is $4$-transitive and hence does not preserve any proper quaternary relation.

\subsubsection*{Acknowledgments}
This project started as a derivative from the work of the second author with Pierre Simon in \cite{itaypierre}. Eventually ideas from that time contributed to the proof of the main theorem here, namely in the proof of (4) implies (1) in Theorem \ref{T:main}. We thank him for allowing us to use his ideas. 
\section{Proofs}

\begin{definition}
Let $f$ be a permutation of $F\cup\{\infty\}$ and $\emptyset\neq O\subseteq F\setminus \{0,1\}$. We say that $f$ is \emph{$O$-preserving} if 
\[[a,b;c,d]\in O\Longleftrightarrow [f(a),f(b);f(c),f(d)]\in O,\] where $a,b,c,d$ are distinct elements from $F\cup\{\infty\}$.
\end{definition}
\begin{remark}
In this paper the cross-ratio is only taken for distinct points so that it takes values in $F\setminus\{0,1\}$ (one can expand the definition to allow repetitions, but this will not be used).
%
\end{remark}

Throughout we will implicitly use the following property of the cross-ratio:
\begin{quote}
For any two quadruples of distinct elements of $F\cup\{\infty\}$, $\{A,B,C,D\}$ and $\{A,B,C,X\}$, the following holds: \[[A,B;C,D]=[A,B;C,X]\Longleftrightarrow D=X.\] 
\end{quote}

\begin{proposition}\label{P:gen-perm-function}
For every $x\in F\setminus\{0,1\}$ there exists a unique function \[g_x:X\to F\cup\{\infty\},\] where $X\subseteq (F\cup\{\infty\})^3$ is the set of triples of distinct elements, such that for every distinct $a,b,c\in F\cup\{\infty\}$ 
\begin{enumerate}
\item $g_x(a,b,c)\neq a,b,c$,
\item $\left[a,b;c,g_x(a,b,c)\right]=x$ and
\item the map $x\mapsto g_x(a,b,c)$ is injective.
\end{enumerate}
Furthermore, if $f$ is an $O$-preserving permutation of $F\cup\{\infty\}$, for some $\emptyset\neq O\subseteq F\setminus\{0,1\}$, then for every distinct $a,b,c\in F\cup\{\infty\}$ there exists a permutation $\alpha:O\to O$, such for every $x\in O$
\[f\left( g_x(a,b,c)\right)=g_{\alpha (x)}\left( f(a),f(b),f(c)\right).\]
\end{proposition}
\begin{proof}
Property $(2)$ uniquely determines $g_x$, and by the definition of the cross-ratio we get the following formula:
\[g_x(a,b,c)=\frac{b(c-a)-ax(c-b)}{(c-a)-x(c-b)}.\] Properties $(1)$ and $(3)$ follow easily.

As for the furthermore, for $x\in O$, define 
\[\alpha(x):=\left[f(a),f(b);f(c),f(g_x(a,b,c))\right]\in O\]
and likewise
\[\alpha^{-1}(x):=\left[a,b;c,f^{-1}(g_x(f(a),f(b),f(c)))\right]\in O.\]
They are both elements of $O$ since $f$ is $O$-preserving, and note that by uniqueness, the definition of $\alpha$ gives that \[f(g_x(a,b,c))=g_{\alpha (x)}(f(a),f(b),f(c)).\]

\begin{claim}
For every $a,b,c,y\in F\cup\{\infty\}$ distinct, if $[a,b;c,y]\in O$ then
\[\alpha([a,b;c,y])=[f(a),f(b);f(c),f(y)]\]
and if $[f(a),f(b);f(c),y]\in O$ then
\[\alpha^{-1}([f(a),f(b);f(c),y])=[a,b;c,f^{-1}(y)].\]
\end{claim}
\begin{claimproof}
Assume that $[a,b;c,y]=x \in O$. By uniqueness 
necessarily \[y=g_x(a,b,c).\] It now follows that
\[ \alpha([a,b;c,y])=\alpha(x)=[f(a),f(b);f(c),f(g_x(a,b,c))]=[f(a),f(b);f(c),f(y)],\] as required.

The proof for $\alpha^{-1}$ is similar.
\end{claimproof}

We may now compute
\[(\alpha\circ\alpha^{-1})(x)=\alpha\left([a,b;c,f^{-1}(g_x(f(a),f(b),f(c)))]\right)=\]
\[[f(a),f(b);f(c),g_x(f(a),f(b),f(c))]=x.\]
Similarly we get that $(\alpha^{-1}\circ \alpha)(x)=x$, as needed.

\end{proof}

\begin{remark}
The previous proposition is obviously also true if we permute the coordinates of the cross-ratio, e.g. consider a function $h_x$ which guarantees that 
\[[a,b;h_x(a,b,c),c]=x.\]
\end{remark}

We will frequently use the following hypothesis.
\begin{hypothesis}\label{H:hyp}
The set $O$ is a \underline{non-empty} subset of $F\setminus \{0,1\}$. The function $f$ is an $O$-preserving permutation of $F\cup\{\infty\}$ which fixes $\{0,1,\infty\}$ pointwise. 

The field $K=k(O)$ is the field generated by the elements of $O$, where $k$ is the prime field.
\end{hypothesis}

\begin{corollary}\label{C:permutation-functions}
Assume Hypothesis \ref{H:hyp}. For all $a\neq b\in F$ there exist permutations $\tau_{a,b},\rho_{a,b},\chi_{a,b},\alpha_{a,b},\beta_{a,b}:O\to O$, such that for every $x\in O$:
\[f(ax+b(1-x))=f(a)\tau_{a,b}(x)+f(b)(1-\tau_{a,b}(x)),\]
\[f\left( \frac{a-(1-x)b}{x}\right)=\frac{f(a)-(1-\rho_{a,b}(x))f(b)}{\rho_{a,b}(x)},\]
\[f\left(\frac{a-xb}{1-x}\right)=\frac{f(a)-\chi_{a,b}(x)f(b)}{1-\chi_{a,b}(x)},\]
\[f\left(\frac{abx-bx-ab+a}{ax-x-b+1}\right)=\frac{f(a)f(b)\alpha_{a,b}(x)-f(b)\alpha_{a,b}(x)-f(a)f(b)+f(a)}{f(a)\alpha_{a,b}(x)-\alpha_{a,b}(x)-f(b)+1},\]
\[f\left(\frac{a-b-abx+bx}{a-b-ax+x}\right)=\frac{f(a)-f(b)-f(a)f(b)\beta_{a,b}(x)+f(b)\beta_{a,b}(x)}{f(a)-f(b)-f(a)\beta_{a,b}(x)+\beta_{a,b}(x)}.\]
(for $\alpha_{a,b}$ and $\beta_{a,b}$ we also require that $a,b\neq 1$.)

Moreover, $f\restriction O$ is a permutation of $O$.

\end{corollary}
\begin{proof}
We apply Proposition \ref{P:gen-perm-function}.
Let $a\neq b\in F$.
For $\tau_{a,b}$ use the identity
\[[ax+b(1-x),a;b,\infty]=x.\]

For $\rho_{a,b}$ use the identity
\[\left[a,\frac{a-(1-x)b}{x};b,\infty\right]=x.\]

For $\chi_{a,b}$ use the identity 
\[\left[a,b;\frac{a-xb}{1-x},\infty\right]=x.\]

For $\alpha_{a,b}$ use the identity
\[\left[ b,a;1,\frac{abx-bx-ab+a}{ax-x-b+1}\right]=x.\]

For $\beta_{a,b}$ use the identity
\[\left[b,1;a,\frac{a-b-abx+bx}{a-b-ax+x}\right]=x.\]

In order to show that $f\restriction O$ is a permutation of $O$, note that $[a,1,0,\infty]=a$ for every $a\in F\setminus \{0,1\}$.

\end{proof}

\begin{lemma}\label{L:first-closures}
Assume Hypothesis \ref{H:hyp}. For every $a,b\in K$ and $x\in O$
\[f(a)+(1-x)f(b)\in f(K) \text{ and}\]
\[xf(a)+f(b)\in f(K).\]
\end{lemma}
\begin{proof}
We start with the first assertion, so let $a,b\in K$ and $x\in O$. If $b=0$ there is nothing to show. If $a=0$ and $b\neq 0$, then since $\tau_{0,b}$ is a permutation, by Corollary \ref{C:permutation-functions},
\[(1-x)f(b)=(1-(\tau_{a,b}\circ\tau_{a,b}^{-1})(x))f(b)=f((1-\tau_{a,b}^{-1}(x))b)\in f(K).\] 
We may thus assume that $a,b\neq 0$ and let $x_2=\rho_{a,0}^{-1}(x)$, so $f(a/x_2)=f(a)/x$. 
If $b=\frac{a}{x_2}$ then \[f(a)+f(b)-xf\left(\frac{a}{x_2}\right)=f(b)\in f(K).\]
Now, assume that $b\neq \frac{a}{x_2}$, and let $x_3=\tau_{a/x_2,b}^{-1}(x)$. Hence
\[f\left(\frac{a}{x_2}x_3+b(1-x_3)\right)=f\left(\frac{a}{x_2}\right)\tau_{a/x_2,b}(x_3)+f(b)(1-\tau_{a/x_2,b}(x_3))\]
\[=f(a)+f(b)-xf(b).\]

Now the second assertion. If $a=0$ there is nothing to show. If $a\neq 0$ and $b=0$ then since $\tau_{a,b}$ is a permutation, $xf(a)\in f(K)$. 
We may thus assume that $a,b\neq 0$ and let $x_2=\chi_{b,0}^{-1}(x)$. If $a=\frac{b}{1-x_2}$ then
\[xf\left(\frac{b}{1-x_2}\right)+f(b)=\frac{x}{1-x}f(b)+f(b)=f(a)\in f(K).\]
Finally, assume that $a\neq \frac{b}{1-x_2}$, and let $x_3=\tau_{a,b/(1-x_2)}^{-1}(x)$. Hence
\[f\left(ax_3+\frac{b}{1-x_2}(1-x_3)\right)=f(a)\tau_{a,b/(1-x_2)}(x_3)+f\left(\frac{b}{1-x_2}\right)(1-\tau_{a,b/(1-x_2)}(x_3))\]
\[=f(a)x+f(b).\]

\end{proof}

\begin{lemma}\label{L:second-closures}
Assume Hypothesis \ref{H:hyp}. For every $0\neq a\in K$ and $x\in O$, 
\[-f(a)^2x+f(a)x+f(a)\in f(K) \text{ and}\]
\[1+x-\frac{x}{f(a)}\in f(K).\]
\end{lemma}
\begin{proof}
Let $x\in O$ and $a\in K$ with $a\neq 0$. If $a=1$ both assertions are trivial, so assume $a\neq 1$.

We start with the first assertion. Since $\tau_{a,0}$ is a permutation, by Corollary \ref{C:permutation-functions}, we may define $x_1=\tau_{a,0}^{-1}(x)$ so, $f(x_1a)=xf(a)$. We aim to use the permutation $\alpha_{a,x_1a}$. Obviously, $x_1a\neq a$ and if $x_1a=1$ then $xf(a)=1$ and $-f(a)^2x+f(a)x+f(a)=0$. Thus by Corollary \ref{C:permutation-functions}, $\alpha_{a,x_1a}$ is a permutation. Let $x_2:=\alpha_{a,x_1a}^{-1}(x)$, so
\[f\left(\frac{a(x_1a)x_2-(x_1a)x_2-a(x_1a)+a}{ax_2-x_2-(x_1a)+1}\right)=\]\[\frac{f(a)f(x_1a)x-f(x_1a)x-f(a)f(x_1a)+f(a)}{f(a)x-x-f(x_1a)+1}=-f(a)^2x+f(a)x+f(a).\]

Now for the second assertion. Since $f\restriction O$ is a permutation, we may define $x_1:=f^{-1}(x)$. We aim to use the permutation $\beta_{a,x_1}$. If $a=x_1$ then the statement is obviously true, so we may assume that $a\neq x_1$ (and both not equal to $1$). By Corollary \ref{C:permutation-functions}, $\beta_{a,x_1}$ is a permutation, so we may define $x_2:=\beta_{a,x_1}^{-1}(x)$ and so
\[f\left(\frac{a-x_1-ax_1x_2+x_1x_2}{a-x_1-ax_2+x_2}\right)=\frac{f(a)-f(x_1)-f(a)f(x_1)x+f(x_1)x}{f(a)-f(x_1)-f(a)x+x}=\]
\[1+x-\frac{x}{f(a)}.\]

\end{proof}

\begin{proposition}\label{P:K-to-K}
Assume Hypothesis \ref{H:hyp}, and if when $\chr(F)=2$ we assume further that $O$ is closed under taking square-roots, then $f(K)=K$.
\end{proposition}
\begin{proof}
We first show that $K\subseteq f(K)$. Note that $O\subseteq f(K)$, indeed $f \restriction O$ is a permutation by Corollary \ref{C:permutation-functions}.

Let $a,b\in K$ and $x\in O$, which exists since $O$ is non-empty. By Lemma \ref{L:first-closures}, $f(a)+f(b)-xf(b)\in f(K)$. By the same lemma
\[f(a)+f(b)=xf(b)+\left( f(a)+f(b)-xf(b)\right)\in f(K).\] 

In order to show that if $f(c)\in f(K)$ then also $-f(c)$, first notice that by considering $\rho_{0,c}$ in Corollary \ref{C:permutation-functions} we see that \[-\frac{1-x}{x}f(c)\in f(K).\]
By considering $\chi_{d,0}$ in the same corollary, we see that \[\frac{f(d)}{1-x}\in f(K),\] for any $d\in K$. In particular $-\frac{1}{x}f(c)\in f(K)$, for any $c\in K$. Similarly, by considering $\tau_{a,0}$, we get that $xf(c)\in f(K)$ for all $c\in K$ and together $-f(c)\in f(K)$ for all such $c$, as needed.

As for the multiplication, by Lemma \ref{L:second-closures}, 
\[-f(a)^2x+f(a)x+f(a)\in f(K),\] for every $x\in O$.
Using the above, and since $f(a),f(a)x\in f(K)$, \[-f(a)^2x\in f(K).\] So $-f(a)^2x=f(b)$ for some $b\in K$. Since $\frac{f(b)}{-x}\in f(K)$ as well, $f(a)^2\in f(K).$

Let $a\neq 0\in K$, by Lemma \ref{L:second-closures}, $1+x-\frac{x}{f(a)}\in f(K)$ for any $x\in O$. Using the above, and since $1,x\in f(K)$, \[\frac{x}{f(a)}\in f(K).\] Using a similar argument to the previous paragraph, $-\frac{1}{f(a)}\in f(K)$, so $\frac{1}{f(a)}\in f(K)$.

Finally, we first assume that $\chr(F)\neq 2$. Let $a,b\in K$. Since $\left( f(a)+f(b)\right)^2\in f(K)$ we get that \[2f(a)f(b)\in f(K).\]
To get that $f(K)$ is a subfield, we need this final claim:
\begin{claim}
If $a\in K$ then $\frac{f(a)}{2}\in f(K)$.
\end{claim}
\begin{claimproof}
We may assume that $a\neq 0$. Since $f(a)\in f(K)$ then $1/f(a)\in f(K)$ and so also $2/f(a)$. Take the inverse again and $f(a)/2\in f(K)$.
\end{claimproof}

We conclude that if $\chr(F)\neq 2$, $f(K)$ is a field and so $K\subseteq f(K)$.

Assume that $\chr(F)=2$. For any $n\geq 0$, let
\[(f(K))^{2^n}:=\{a^{2^n}:a\in f(K)\}.\] Note that since $f(K)$ is closed under squares, $\{(f(K))^{2^n}:n\geq 0\}$ forms a decreasing sequence under inclusion. The following is an easy observation:
\begin{claim}
For every $n\geq 0$, $(f(K))^{2^n}$ is also closed under addition, additive and multiplicative inverses and taking square powers.
\end{claim}

Consider $L:=\bigcap_{n\geq 0}(f(K))^{2^n}$. 
\begin{claim}
$L$ is a field containing $O$. Therefore, $K\subseteq L\subseteq f(K)$.
\end{claim}
\begin{claimproof}
Since $O$ is closed under square-roots it is contained in $L$ and by the previous claim $L$ is closed under addition, additive and multiplicative inverses and taking square-powers. Let $a,b\in L$ and let $n\geq 0$. We will show that $ab\in (f(K))^{2^n}$. Since $a,b\in (f(K))^{2^{n+1}}$, by the following form of Hua's identity (first mentioned in \cite{hua} but we use the more manageable form from \cite[page 2]{jacobson}):

\[a-(a^{-1}+(b^{-2}-a)^{-1})^{-1}=a^2b^2,\]
and by the last claim, $a^2b^2\in (f(K))^{2^{n+1}}$. Since the Frobenius map is injective, $ab\in (f(K))^{2^n}$, as required.
\end{claimproof}

Either way, $K\subseteq f(K)$, but since $f^{-1}$ is also $O$-preserving by definition, we actually have $K\subseteq f(K)\subseteq K$ as required.
\end{proof}

\begin{definition}
For an $Aut(F)$-invariant subfield $K\subseteq F$,  a \emph{$K$-chain} is an image of $K\cup \{\infty\}$ under the action of $P\Gamma L_2 (F)$.
\end{definition}
\begin{remark}
The term \emph{$K$-chain} is due originally to von Staudt who introduced it for any real subline of the complex projective line. Note that the usual definition of $K$-chain is for any subfield $K\subseteq F$ and only images under the action of $PGL_2(F)$ (see \cite[Definition 2.2.2]{chain}). Naturally, these definitions are equivalent for $Aut(F)$-invariant subfields. See \cite{chain} for more on this subject, and in a higher level of generality.
\end{remark}

We recall that an action of a group $G$ on a set $X$ (with $|X|\geq k$) is \emph{$k$-transitive} if $G$ acts transitively on the set of $k$-tuples of distinct elements of $X$. For example, the action of the group $PGL_2(F)$ on $F\cup \{\infty\}$ is $3$-transitive.

\begin{corollary}\label{C:sending K-lines to K-lines}

Let $f$ be an $O$-preserving permutation of $F\cup\{\infty\}$, for some non-empty $O\subseteq F\setminus \{0,1\}$ which is $Aut(F)$-invariant. If, when $\chr(F)= 2$, we further assume $O$ is closed under taking square-roots, then $f$ sends $K$-chains to $K$-chains, where $K=\primef(O)$.

\end{corollary}
\begin{proof}
Let $X=T(K\cup\{\infty\})$ be a $K$-chain, where $T\in P\Gamma L_2(F)$. Since $O$ is $Aut(F)$-invariant, $f\circ T$ is also $O$-preserving. Thus we may assume that $X=K\cup\{\infty\}$. Since $PGL_2(F)$ is 3-transitive and preserves the cross-ratio, by composing with an element of $PGL_2(F)$ we may assume that $f$ fixes $\{0,1,\infty\}$ pointwise. By Proposition \ref{P:K-to-K}, $f(K\cup \{\infty\})=K\cup\{\infty\}$ as needed. 
\end{proof}

We recall some definitions from affine geometry. Given a $K$-vector space $V$, an \emph{affine line} is a set of the form $Ka+b$ for some $a,b\in V$ with $a\neq 0$. A map $T:V\to V$ is called \emph{semilinear} if there exists a field automorphism $\sigma \in Aut(K)$ such that for all $v,u\in V$ and $x,y\in K$
\[T(xv+yu)=\sigma(x)T(v)+\sigma(y)T(u).\]

\begin{fact}[The Fundamental Theorem of Affine Geometry]\cite[Theorem 3.5.6]{fundamentalaffine}\label{F:fundamentalaffine}
Let $K$ be any field and $V$ a $K$-vector space of dimension at least $2$. If $f$ a permutation of $V$ sending affine lines to affine lines then there exists a semilinear map $T:V\to V$ and $b\in V$ such that $f(x)=T(x)+b$ for all $x\in V$.
\end{fact}

The following theorem is a generalization of the main theorem in \cite{hoffman}.

\begin{theorem}\label{T:main}
Let $\primef \subsetneq F$ be a field, where $\primef$ is the prime field of $F$, and let $f$ be a permutation of $F\cup\{\infty\}$. 

If $\chr(F)=2$ we assume further that 
\begin{enumerate}[(a)]
\item $F$ is perfect and
\item $|F|>4$.
\end{enumerate}

Then the following are equivalent
\begin{enumerate}
\item $f\in P\Gamma L_2(F)$.
\item $f$ is $O$-preserving for all non-empty $O\subseteq F\setminus \{0,1\}$ which is $Aut(F)$-invariant and satisfies $k(O)\subsetneq F$.

\item $f$ is $O$-preserving for some non-empty $O\subseteq F\setminus \{0,1\}$ which is $Aut(F)$-invariant and satisfies $k(O)\subsetneq F$.

\item There exists an $Aut(F)$-invariant subfield $K\subsetneq F$ such that $f$ sends $K$-chains to $K$-chains.
\item For all $Aut(F)$-invariant subfields $K\subsetneq F$, $f$ sends $K$-chains to $K$-chains.
\end{enumerate}
\end{theorem}
\begin{remark}
\begin{enumerate}
\item Since we are assuming $O$ is non-empty, if we drop the assumption that $|F|>4$ then $(1)$ might be true even if $(3)$ is not. For example, for $F=\mathbb{F}_4$, the field with $4$ elements, there are no proper intermediate fields between $\mathbb{F}_2$ and $\mathbb{F}_4$, so $(3)$ is not true.
\item The implication $(4)\Rightarrow (1)$ is well known for fields $F$ with $\chr(F)\neq 2$, see \cite[Theorem 9.2.5]{chain}, but we provide a direct proof.
\end{enumerate}
\end{remark}
\begin{proof}
$(1)\Rightarrow (2)$. This is by the definition of $P\Gamma L_2(F)$.

$(2)\Rightarrow (3)$. If $\chr(F)\neq 2$ just take $O=k\setminus \{0,1\}$ (which is non-empty). If $\chr(F)=2$, take elements generating the subfield of $4$ elements.

$(3)\Rightarrow (4)$. Let $K=\primef(O)$. Since $O$ is $Aut(F)$-invariant, $K$ is $Aut(F)$-invariant and if $\chr(F)=2$ then $O$ is closed under taking square-roots since $F$ is perfect and hence the inverse of the Frobenius map is an automorphism. Now apply Corollary \ref{C:sending K-lines to K-lines}.

$(4)\Rightarrow (1)$. By composing with an element of $PGL_2(F)$, we may assume that $f$ fixes $\{0,1,\infty\}$ pointwise. We plan to use Fact \ref{F:fundamentalaffine}, so we must show that, in $F$ as a $K$-vector space, $f\restriction F$ sends affine lines to affine lines. Since $f$ fixes $\{\infty\}$, and sends $K$-chains to $K$-chains, it is sufficient to show the following, where by a \emph{projective affine line} we mean a union of an affine line with $\{\infty\}$,

Recall also that any $K$-chain is equal to some $T(K\cup\{\infty\})$, where $T(x)$ is of the form \[\frac{ax^\sigma+b}{cx^\sigma+d},\] for $a,b,c,d\in F$, with $ad-bc\neq 0$, and $\sigma\in Aut(F)$. Since $K$ is $Aut(F)$-invariant we may assume that $\sigma=id$. 

\begin{claim}
A subset of $F\cup\{\infty\}$ is a $K$-chain which includes $\infty$ if and only if it is a projective affine line.
\end{claim}
\begin{claimproof}
A projective affine line has the form $a(K\cup\{\infty\})+b$, for $a,b\in F$, so it is a $K$-chain.
For the other direction, by translation it is enough to show that any $K$-chain containing $0$ and $\infty$ is a projective affine line.

Note that projective affine lines that contain $0$ are just of the form $aK$ (for $a\neq 0$), and that both families of projective affine lines containing $0$ and $K$-chains containing $0$ and $\infty$ are closed under scalar multiplication (by non-zero elements from $F$) and inverse ($x\mapsto 1/x$). So it is enough to show that after applying finitely many operations of the form above on a $K$-chain containing $0$ and $\infty$ gives a projective affine line (containing $0$). Assume that we are given a $K$-chain of the form $T(K\cup\{\infty\})$ for $T$ as above which contains $0$ and $\infty$.

\begin{list}{•}{}
\item If $T$ is of the form $ax+b$ with $a\neq 0$, then $b=0$ and we are done.
\item If $T$ is of the form $b/(cx+f)$ with $c\neq 0$, then we are done by the first bullet (after dividing by $b$ and taking inverse).
\item If the $T$ is of the form $\frac{ax+b}{cx+d}$ with $a,c\neq 0$, then after multiplying by $c/a$, we may assume that $a=c=1$, and then since the chain contains $0$, $b\in K$, and since the chain contains $\infty$, $d\in K$. So it is equal to $(K\cup\{\infty\})$. 
\end{list}
%
%
\end{claimproof}

As a result, $f$ preserves the system of affine lines in the  $K$-vector space $F$. Since $K\subsetneq F$, $\dim_K F\geq 2$ so by the fundamental theorem of affine geometry (Fact \ref{F:fundamentalaffine}) and since $f(0)=0$, $f$ must be additive and so also $f(-a)=-f(a)$ for all $a\in F$.

The conjugation of $f$ by the $PGL_2(F)$ map $x\mapsto 1/x$ also satisfies the above, so it is also additive. This translates to \[\frac{f(a)f(b)}{f(a)+f(b)}=f\left( \frac{ab}{a+b}\right),\] for all nonzero $a,b\in F$.
By setting in the equation $a=1$ and $b=t-1$ (for $t\neq 1$) we obtain, $f(t)f(t^{-1})=1$, thus $f$ commutes with inversion.

Putting in the same equation $b=1-a$, for $a\neq 0,1$, we obtain $f(a)f(1-a)=f(a(1-a))$, which gives $f(a^2)=f(a)^2$. 

If $\chr(F)\neq 2$ then, since $f$ is additive, $f(x/2)=f(x)/2$ for all $x\in F$. Set $a=x+y$ in the last equation to get $f(xy)=f(x)f(y)$ for all $x,y\in F$.

If $\chr(F)=2$ we once again use Hua's identity:
\[f(a^2b^2)=f\left(a-(a^{-1}+(b^{-2}-a^{-1})^{-1}\right)\]
\[=f(a)-(f(a)^{-1}+(f(b)^{-2}-f(a)^{-1})^{-1}=f(a^2)f(b^2).\]
Hence $f(ab)^2=\left( f(a)f(b)\right)^2$, so $f(ab)=f(a)f(b)$.

Either way, we get that $f$ is an automorphism of $F$, an in particular $f\in P\Gamma L_2(F)$.

$1\Rightarrow 5$ is clear and $5\Rightarrow 4$ follows by taking $K=k$.
\end{proof}

As a direct corollary of Theorem \ref{T:main} we get the following.

\begin{corollary}\label{C:main}
Let $F$ be a field, $\primef$ its prime field and $\emptyset\neq O\subseteq F\setminus \{0,1\}$ which is $Aut(F)$-invariant. If
\begin{enumerate}
\item $\primef(O)\subsetneq F$ and
\item if $\chr(F)=2$ then $F$ is perfect and $|F|>4$,
\end{enumerate}
then the subgroup of $O$-preserving permutations of $F\cup \{\infty\}$ is exactly $P\Gamma L_2(F)$.
\end{corollary}

It was shown by Hoffman in \cite{hoffman}, that if $F$ is a field, $a\in F\setminus \{0,1\}$ and $f$ is an $\{a\}$-preserving permutation of $F\cup\{\infty\}$ then $f\in P\Gamma L_2(F)$. One may ask, what about if $f$ preserves a set of cardinality larger than $1$? Theorem \ref{T:main} only gives a partial answer. More specifically, can the assumption $k(O)\subsetneq F$ be dropped? For example:
\begin{question}
Does the subgroup of permutations of $\mathbb{Q}(\sqrt{2})\cup\{\infty\}$ which are $\{\pm\sqrt{2}\}$-preserving properly contain $P\Gamma L_2(\mathbb{Q}(\sqrt{2}))$? 
\end{question}

\section{Every proper extension of $P\Gamma L_2(F)$ is $4$-transitive}
Our final aim is to show that, as a corollary of Theorem \ref{T:main}, any group of permutations of $F\cup\{\infty\}$, for $F$ algebraically closed of transcendence degree at least $1$, which properly contains $P\Gamma L_2(F)$ must be $4$-transitive and as a result does not preserve any non-trivial $4$-relation.

\begin{lemma}\label{L:orbits}
Let $\{O_i\}_{i\in I}$ be the orbits of $Aut(F)$ acting on $F\setminus \{0,1\}$. Then the orbits of $P\Gamma L_2(F)$ acting on quadruples of distinct elements from $F\cup \{\infty\}$ are
\[ \{(a,b,c,d): [a,b;c,d]\in O_i\}_{i\in I}.\]
\end{lemma}
\begin{proof}
Let $T\circ \sigma\in P\Gamma L_2(F,)$, for $T\in PGL_2(F)$ and $\sigma\in Aut(F)$. Since elements of $PGL_2(K)$ preserve the cross-ratio, and $\sigma(O_i)=O_i$ by definition, $P\Gamma L_2(F)$ preserves the orbits.

Now, let $(x,y,z,w), (a,b,c,d)$ be quadruples of distinct elements such that \[[x,y;z,w],[a,b;c,d]\in O_i.\] By applying an element of $Aut(F)$ we may assume that $[x,y;z,w]=[a,b;c,d]$. 

Since $PGL_2(F)$ is $3$-transitive, there exists $T\in PGL_2(K)$ such that $T(x)=a, T(y)=b, T(z)=c$. So we have that \[[a,b;c,T(w)]=[a,b;c,d].\]
As $a,b,c$ are distinct we have that $T(w)=d$. 
\end{proof}

\begin{theorem}
Let $F$ be an algebraically closed field of transcendence degree at least $1$, and $H$ be a group of permutations of $F\cup\{\infty\}$ properly containing $P\Gamma L_2(F)$. Then $H$ is $4$-transitive.
\end{theorem}
\begin{proof}
By Lemma \ref{L:orbits}, the action of $P\Gamma L_2 (F)$ breaks the space of quadruples of distinct elements from $F\cup\{\infty\}$ into infinitely many finite orbits (corresponding to finite Galois orbits) and one infinite orbit (corresponding to the Galois orbit of transcendentals).

Thus it is enough to show that every orbit of the action of $H$ on the space of quadruples of distinct elements from $F\cup\{\infty\}$ intersects the orbit corresponding to the transcendentals.

Aiming for a contradiction, assume there exists an orbit $X$ of the action of $H$, which only contains orbits with algebraic cross-ratio, i.e. \[X=\bigcup_{i\in I_0} \{(a,b,c,d):[a,b;c,d]\in O_i\},\] for some $I_0\subseteq I$ and $O_i$ finite, where $I$ and $O_i$ are as in Lemma \ref{L:orbits}. Let $O=\bigcup_{i\in I_0} O_i$ be the cross-ratios arising from quadruples from $X$ and let $K=\primef(O)$, where $\primef$ is the prime field. Note that $K\subsetneq F$ and that $O$ is $Aut(F)$ invariant. By assumption every element of $H$ is $O$-preserving and thus by Corollary \ref{C:main}, $H\subseteq P\Gamma L_2(F)$, contradiction.

\end{proof}

\begin{question}
What about other fields? For instance, is it true that every group of permutations of $\mathbb{Q}(\sqrt{2})\cup\{\infty\}$ properly containing $P\Gamma L_2(\mathbb{Q}(\sqrt{2}))$ must be $4$-transitive?
\end{question}
\bibliographystyle{alpha}
\bibliography{PGammaL2}

\end{document}